\documentclass[10pt]{amsart}
\usepackage{amssymb,amscd, graphicx, mathtools}
\usepackage{hyperref}
\usepackage{color}
\usepackage[margin=1.5in]{geometry}

\def \c{\mathbb{C}}

\def \r{\mathbb{R}}

\def \p{\mathbb{P}}

\def \.{\cdot}

\def \Lie{\textup{Lie}}

\theoremstyle{plain}
\newtheorem{Th}{Theorem}[section]

\newtheorem{Prop}[Th]{Proposition}
\newtheorem{Cor}[Th]{Corollary}

\theoremstyle{definition}
\newtheorem{Ex}[Th]{Example}
\newtheorem{Def}[Th]{Definition}

\begin{document}
\title{Momentum graphs, Chinese remainder theorem and the surjectivity of the restriction map}
\author{James B. Carrell\\ Kiumars Kaveh}

\begin{abstract}
We consider a Chinese remainder theorem for (labeled) graphs. For $X$ a GKM $T$-variety and $Y$ an invariant subvariety, we use this to give a condition for surjectivity of the restriction map $H^*(X) \to H^*(Y)$. In particular, this applies to certain invariant subvarieties in a smooth toric variety.
\end{abstract}

\thanks{The second author is partially supported by a National Science Foundation Grant (Grant ID: DMS-1601303).}



\maketitle
\section{Introduction}
Throughout the base field is $\c$. Let $X$ be a smooth projective variety equipped with an algebraic action of a torus $T$ such that $(X, T)$ is equivariantly formal and moreover has finitely many fixed points and finitely many $T$-invariant curves. Such an $X$ is usually called a \emph{GKM $T$-variety} after Goresky, Kottwitz and MacPherson \cite{GKM}. There is a beautiful and purely combinatorial description of the cohomology ring $H^*(X, \c)$ in terms of the graph of fixed points and invariant curves. Many important classes of varieties such as toric varieties, flag varieties and more generally spherical varieties are examples of GKM varieties. 

In this note we address a natural question:
If $Y$ is a $T$-invariant closed
subvariety of $X$, when is the restriction map $i^*: H^*_T(X, \c) \to H^*_T(Y, \c)$ induced by the inclusion $i: Y \hookrightarrow X$ surjective? We consider a variant of the Chinese remainder theorem to give a sufficient conditions for this to happen. This in particular gives an easy to verify condition in the case when $X$ is a toric variety. Our surjectivity criterion is also applicable to torus invariant subvarieties in a flag variety such as Richardson varieties.   

The fact that the existence of a group action can yield surjectivity results has already
been noticed in the case $X$ is a smooth projective variety over $\c$ admitting an
action of the triangular group $\mathcal{B}$ of $2\times 2$ upper triangular matrices
such that $X^{\mathcal{B}}$ is a point. In this case, 
for any smooth $\mathcal{B}$-stable 
subvariety $Y$ of $X$, the restriction map
$H^*(X,\c)\to H^*(Y,\c)$ is surjective  (see  \cite{TG}).

\section{Cohomology of a graph}
Fix a natural number $n$ and let $R$ be the polynomial ring $\c[x_1,\ldots, x_n]$.
\begin{Def} A {\em labeled graph} is a simple graph $\Gamma$ with vertex set $V$ and edge set $E$. 
For $v,w\in V$,
a directed edge from $v$ to $w$ is denoted $(v,w)$. 
We require that each directed edge $(v,w)$ is labelled by a homogeneous linear polynomial $e(v,w)\in R$, and furthermore, we put $e(v,w)=-e(w,v)$.   
\end{Def}
 
\begin{Def}
Let $\Gamma = (V, E)$ be a labeled graph. The cohomology of the graph
$\Gamma$ denoted by $H^*(\Gamma)$ is the set of all maps $f: V \to
R$ such that for every labeled edge $(v,w)$,
$$f(v) \equiv f(w) ~(mod~ e(v,w)),$$ that is, $f(v) - f(w)$ is
divisible by $e(v,w)$.
\end{Def} 

\section{The GKM-graph of a torus action}
Let $X$ be a compact $d$-dimensional complex manifold equipped with
a faithful action of a complex torus $T$ of dimension $n$.

\begin{Def}
We say that $X$ is a GKM-manifold if it has the following properties:
\begin{itemize}
\item[(i)] $X^T$ is finite.
\item[(ii)] For every $x \in X^T$, the isotropy representation of $T$ on $T_xX$
 has distinct weights $w_{i,p} \in \Lie(T)^*, i = 1, \ldots ,d$, which  are pairwise
linearly independent.
\end{itemize} 
\end{Def}

Let $E(X)$ denote the set of all $T$-invariant curves in $X$. 
Each $C\in E(X)$ is isomorphic to a copy of
$\c\p^1$, so $C^T=\{v,w\}$ for a unique $v\ne w$. It is well known that if the weight of the 
isotropy representation of $T$ on $T_vC$ is $\omega \in \Lie(T)^*$, then the weight of the
isotropy representation on $T_wC$ is $-\omega$.
Choosing a one parameter subgroup $\lambda$ of $T$
whose fixed point set is $X^T$,
one may define a direction on $C$ by
designating $C$'s initial point to be $\lim _{a\to 0} \lambda(a) z=v$
for some (or any) $z \in C -C^T$. For the next definition, fix a one parameter subgroup $\lambda$ as above.
 
\begin{Def}
Let $X$ be a GKM $T$-manifold. The GKM graph of $X$, denoted by
$\Gamma_X$, is the labeled graph whose vertices are the fixed points
and the edges are the $T$-invariant curves directed by $\lambda$ as above. The label assigned to an
invariant curve $C(v,w)$ with initial point $v$  is the weight $\omega\in \Lie(T)^*$
of the isotropy action of $T$ defined above.
\end{Def}
One can show that $\Gamma_X$ is a $d$-regular graph.

The following gives a combinatorial description of the equivariant cohomology of a GKM manifold \cite{GKM}. 
\begin{Th}[Goresky-Kottwitz-MacPherson] 
Let $X$ be a GKM $T$-manifold. Then the equivariant cohomology algebra $H^*_T(X, \c)$ is isomorphic to the graph cohomology $H^*(\Gamma_X)$. 
\end{Th}

In many cases, the graph $\Gamma_X$ of a GKM manifold $X$ can be realized as sitting in a Euclidean space (namely $\Lie(T)^*$) with vertices as points and labels as vectors in the same directions as lines joining the vertices.  
\begin{Prop}    \label{prop-moment-graph}
Let $X$ be a K\"{a}hler GKM $T$-manifold. Also assume that the
action is Hamiltonian. Then the moment map $\mu: X \to
\Lie(T)^*$ maps an edge connecting two vertices $v,w \in X^T$ to the
line segment joining $\mu(v), \mu(w) \in \Lie(T)^*$. Moreover, the
label $e(v,w)$ is an integral vector in the direction of this
line segment.
\end{Prop}

\section{A Chinese remainder theorem for subgraphs}
We will need the following version of the Chinese remainder theorem. The proof is very similar to the usual Chinese remainder theorem in a polynomial ring (cf. \cite[Section 7.6]{DF}). 
\begin{Th} \label{th-CRT}
Let $e_1, \ldots, e_m~ (m \leq n)$ be linearly independent linear
polynomials in $R = \c[x_1, \ldots, x_n]$. Also let $a_1, \ldots,
a_m$ be polynomials in $R$. Then the system of congruences
\begin{align*} \label{for-CRT}
f &\equiv a_1 ~(mod~ e_1)\cr &... \cr f &\equiv a_m ~(mod~ e_m)
\end{align*}
has a solution $f \in R$ if and only if for every $i, j$ ($1 \leq i,
j \leq m$) we have $a_i - a_j \in \langle e_i, e_j \rangle$ (where
$\langle e_i,e_j \rangle$ denotes the ideal in $R$ generated by the
$e_i$ and $e_j$).
\end{Th}

Now let $\Gamma$ be a connected labeled graph. A subgraph
$\Gamma' = (V', E')$ is a called the {\it subgraph induced by} $V'$
if whenever two vertices of $v,w \in V'$ are adjacent in $\Gamma$ then they are also adjacent in $\Gamma'$ with the same label
$e(v,w)$.
Let $g:V'\to R$ be an element of $H^*(\Gamma')$. We
wish to extend $g$ to an element of $H^*(\Gamma)$, i.e. find an
element $f:V \to R$ in $H^*(\Gamma)$ such that $f(v) = g(v)$
whenever $v \in \Gamma'$. We call $f$ a solution to the Chinese remainder problem for the triple $(\Gamma, \Gamma', g)$.

\begin{Prop} \label{prop-CRT-subgraph}
Let $\Gamma$ be a connected labeled graph and let $\Gamma' = \Gamma
- \{v\}$ be an induced subgraph which is missing only one vertex.
That is, $V = V' \cup \{v\}$. Then for a cohomology class $g \in
H^*(\Gamma')$ the Chinese remainder problem $(\Gamma, \Gamma', g)$
has a solution if and only if for any two vertices $v_1,v_2 \in V'$
which are adjacent to $v$ with labels $e_1$ and $e_2$
respectively, we have $g(v_1) - g(v_2) \in \langle e_1, e_2
\rangle$.
\end{Prop}
\begin{proof}
The proposition follows from Theorem \ref{th-CRT}.
\end{proof}

We say $(\Gamma, \Gamma')$ is a \emph{Chinese remainder pair} if for any $g \in H^*(\Gamma')$, the Chinese remainder problem for $(\Gamma, \Gamma', g)$ has a solution.
Suppose there exists a sequence of induced subgraphs $$\Gamma' = \Gamma_0 \subset \cdots \subset \Gamma_r = \Gamma$$ such that $(\Gamma_i, \Gamma_{i-1})$ is a Chinese remainder pair, for $i=1, \ldots, r$. Then the restriction map $H^*(\Gamma) \to H^*(\Gamma')$ is surjective. 

Suppose $\Gamma$ is a labeled graph that is embedded in some Euclidean space $\r^n$. That is, the set of vertices of $\Gamma$ is a finite subset of $\r^n$ and if two vertices are connected the label $e(v, w)$ is a multiple of the vector $w-v$ (see Proposition \ref{prop-moment-graph}). 

\begin{Prop}  \label{prop-Euclidean-graph}
Let $\Gamma$ be as above. Let $\Gamma' = \Gamma \setminus \{v\}$ be an induced subgraph which is missing only one vertex. Suppose for any two vertices $v_1$, $v_2$ in $\Gamma'$ that are adjacent to $v$ we know that $v_1$ can be joined to $v_2$ with a sequence of edges that lie in the plane through $v$, $v_1$ and $v_2$. Then $(\Gamma, \Gamma')$ is a Chinese remainder pair.
\end{Prop}
\begin{proof}
Let $g \in H^*(\Gamma')$ and let $g_1$ and $g_2$ be the
polynomials in $g$ attached to the vertices $v_1$ and $v_2$ respectively. By Proposition \ref{prop-CRT-subgraph} we need to show that $g_1 - g_2$
is in $\langle e_1,e_2 \rangle$ where $e_i$ is the label of the edge $(v, v_i)$, $i=1, 2$. By assumption $v_1$ and
$v_2$ are connected by a path of edges in the plane through $v$, $v_1$ and $v_2$. Let us denote the labels of these edges by $l_1, \ldots, l_r$. Since $g \in H^*(\Gamma')$ it follows that $v_1 - v_2 \in \langle l_1, \ldots, l_r \rangle \subset \langle e_1, e_2 \rangle$ and therefore $g_1 - g_2 \in \langle e_1, e_2
\rangle$. This proves the proposition.
\end{proof}

\section{Example: graph of a polytope and toric varieties}
Let $X_\Sigma$ be an $n$-dimensional smooth projective toric variety with associated complete fan $\Sigma$. Let $\Delta$ be
a polytope which is normal to the fan $\Sigma$. Then $\Delta$ is a simple polytope, i.e. the cone at each vertex is a simplicial cone. One knows that the variety $X_\Sigma$ is a GKM $T$-variety and its corresponding graph $\Gamma_{X_\Sigma}$ can be identified with the graph 
$\Gamma_\Delta$ corresponding to (the $1$-skeleton of) the polytope $\Delta$. For a polytope $\Delta$ its corresponding labeled graph $\Gamma_\Delta$ is the graph whose vertices are the vertices of
$\Delta$ and two vertices $v$ and $w$ are connected if they are
connected by an edge in the polytope and the label $e(v,w)$ is the
shortest integral vector in the direction of the vector $w-v$.

We would like to see for which subgraphs $\Gamma'$, we have $(\Gamma_\Delta, \Gamma')$ is a Chinese remainder pair. The following example demonstrates that this is not always the case.
\begin{Ex}
Let $\Delta$ be the unit cube in $\r^3$. Let $\Gamma'$ be the
subgraph induced by the vertices $v_1 = (0,0,0), v_2 = (0,0,1),
v_3=(0,1,1), v_4=(1,1,1), v_5=(1,1,0)$. It is a connected subgraph.
Now consider the collection of polynomials $g_i \in \c[x,y,z]$
defined by $g_1 = g_2 = g_3 = g_4 = x$ and $g_5 = x+z$. One can
check that if we assign $g_i$ to the vertex $v_i$ then $g = (g_1,
\ldots, g_5)$ defines an element of $H^*(\Gamma')$. From Theorem
\ref{th-CRT}, if the Chinese remainder problem $(\Gamma_\Delta,
\Gamma', g)$ has a solution then $g_5 - g_1$ must lie in the ideal
generated by $x$ and $y$. But $g_5 - g_1 = z \notin \langle x,y
\rangle$. This shows that the Chinese remainder problem
$(\Gamma_\Delta, \Gamma', g)$ does not have a solution.
\begin{figure}[h]
\includegraphics[width=6cm]{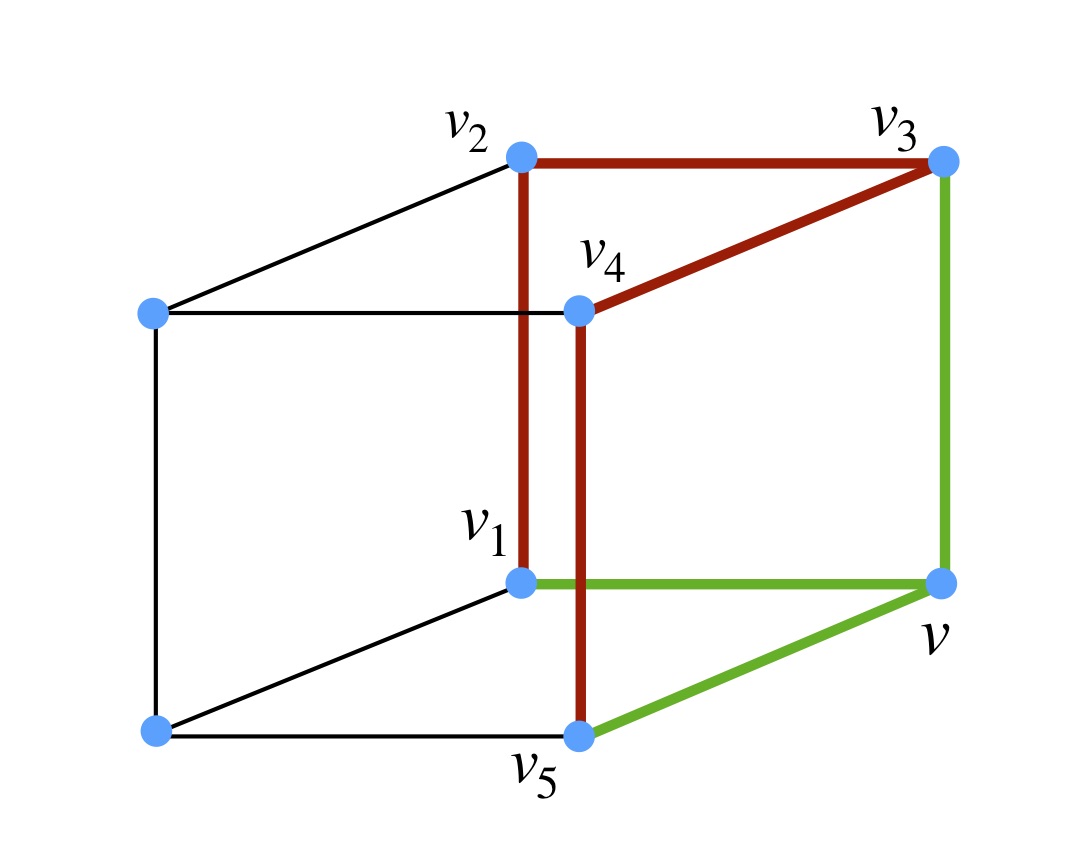}
\end{figure}
\end{Ex}

\begin{Def}
Let $\Gamma'$ be an induced subgraph of $\Gamma_\Delta$. We say that $\Gamma'$ is \emph{$2$-face connected} if the intersection of $\Gamma'$ with any $2$-dimensional face of $\Delta$ is connected. 
\end{Def}

\begin{Prop} Let $\Gamma_1 \subset \Gamma_2$ be induced subgraphs of $\Gamma_\Delta$ such that $\Gamma_2$ is obtained from $\Gamma_1$ by adjoining one more vertex $v$. Moreover assume that $\Gamma_2$ is $2$-face connected. Then the restriction map $H^*(\Gamma_2) \to H^*(\Gamma_1)$ is surjective. 
\end{Prop}
\begin{proof}
The proposition follows immediately from the definition of $2$-face connected subgraph and Proposition \ref{prop-Euclidean-graph}. 
\end{proof}

\begin{Cor}   \label{cor-seq-vertices}
Let $\Gamma' = (V', E')$ be an induced subgraph of $\Gamma_\Delta$. Let $v_1, \ldots, v_r$ be a sequence of vertices in $\Gamma_\Delta$ and let $\Gamma_i$ be the subgraph induced by $V' \cup \{v_1, \ldots, v_i\}$ with $\Gamma_0 = \Gamma'$ and $\Gamma_r = \Gamma_\Delta$. Suppose for $i=0, \ldots, r-1$,  $\Gamma_i$ is $2$-face connected. Then the restriction map $H^*(\Gamma_\Delta) \to H^*(\Gamma')$ is surjective.
\end{Cor}

Let $\xi$ be a $1$-parameter subgroup of $T$. We regard $\xi$ as a linear function on the vector space $M_\r = M \otimes \r$ where $M$ is the lattice of characters of $T$. We assume $\xi$ is in
general position with respect to $\Delta$, i.e. $\xi$ is not orthogonal to any face of $\Delta$. The linear function $\xi$ is analogue of a Morse function on $\Delta$. 
\begin{Cor}
Fix a constant number $c$. Let $Y$ be a $T$-invariant subvariety of $X_\Sigma$ whose GKM graph $\Gamma_Y$ is induced by all the vertices $v$ with $\xi(v)
\leq c$. Then $\Gamma_Y$ satisfies the condition in Corollary \ref{cor-seq-vertices} and hence the restriction map $H^*_T(X_\Sigma) \to H^*_T(Y)$ is
surjective.
\end{Cor}
\begin{proof}
Let $\{w_1, \ldots, w_r\}$ be vertices not in $\Gamma_Y$ and ordered such that $\xi(w_1) < \cdots < \xi(w_r)$. For each $i$ let $\Gamma_i$ be the graph induced by $\Gamma_Y$ and $\{v_1, \ldots, v_i\}$. Since $\Delta$ is
convex, the intersection of $\Delta$ and the half-space
$\xi^{-1}((-\infty, c])$ is connected. Thus the graph $\Gamma_i$ induced by the vertices $v$ with $\xi(v) \leq \xi(v_i)$ is also connected. As for
the $2$-face connectedness of $\Gamma_i$, let $u_1, u_2 \in \Gamma_i$ belong to the same $2$-dimensional face $\sigma$. We want to show that $u_1$ and $u_2$ are connected by a path in $\Gamma_i \cap
\sigma$. Let $P$ be the $2$-dimensional plane through the face $\sigma$. The set of points $x \in P$ defined by the inequality $\xi(x) \leq \xi(v_i)$ is a half-plane $H$ in $P$. There are two paths in
$\sigma$ joining $u_1$ and $u_2$. Since $u_1,u_2$ are in $H$ and $\sigma$ is convex, one of these paths should completely lie in $H$. This proves the proposition.
\end{proof}

\section{Example: Bruhat graph and the flag variety}
Let $G$ be a connected reductive algebraic group with Borel subgroup  $B$ and  maximal torus 
$T\subset B$. It is well-known that the flag variety $G/B$ is a GKM $T$-variety with the left action of $T$. The GKM graph $\Gamma$ of $G/B$ is identical to the Bruhat graph $\Gamma_G$ of $(G,T)$. As a labelled graph,  $\Gamma_G$ has vertex set  
the Weyl group $W$ and directed edges corresponding to ordered pairs $(u,w)\in W\times W$
where $uw^{-1}$ is a reflection $r$ and $\ell(u)< \ell(w)$. If $r$ is the reflection
corresponding to the positive root $\alpha$, then the directed edge has label $\alpha$. 
We will denote this edge by $u \xrightarrow{\alpha} w$

Let $Y$ be a $T$-invariant subvariety of $G/B$ whose GKM graph $\Gamma' = (W', E')$ is an induced subgraph of the Bruhat graph $\Gamma = \Gamma_G$. Let $W \setminus W' = \{w_1, \ldots w_r\}$ and let $\Gamma_i$ be the graph induced by $W' \cup \{w_1, \ldots, w_i\}$ with $\Gamma_0 = \Gamma'$ and $\Gamma_r = \Gamma$. 

The following is an immediate corollary of Proposition \ref{prop-Euclidean-graph}. 
\begin{Cor}   \label{cor-Bruhat-graph}
With notation as above, suppose the following hold for every $i=1, \ldots r$: whenever $u_1, u_2 \in W_{i-1}$ with $u_1 \xrightarrow{\alpha_1} w_i$ and $u_2  \xrightarrow{\alpha_2} w_i$, there is $w' \in W_{i-1}$ such that $w' \xrightarrow{\alpha'_1} u_1$ and $w' \xrightarrow{\alpha'_2} u_2$ where $\alpha'_1, \alpha'_2$ are in the linear span of $\alpha_1,  \alpha_2$. Then the restriction map $i^*: H^*(G/B, \c) \to H^*(Y, \c)$ is surjective. 
\end{Cor}

From Corollary \ref{cor-Bruhat-graph} one can recover the known facts that the restriction map from the cohomology of $G/B$ to a Richardson variety (in particular a Schubert variety) is surjective. 






\begin{thebibliography}{20}
\bibitem[DF04]{DF} Dummit, D. S.; Foote, R. M. \emph{Abstract algebra}. Third edition. John Wiley \& Sons, Inc., Hoboken, NJ, 2004.
\bibitem[GKM98]{GKM} Goresky, Mark; Kottwitz, Robert; MacPherson, Robert Equivariant cohomology, Koszul duality, and the localization theorem. Invent. Math. 131 (1998), no. 1, 25--83.
\bibitem[CK08]{TG} Carrell, J. B.; Kaveh, K. \emph{On the equivariant cohomology of subvarieties of a $\mathfrak{B}$-regular variety}. Transform. Groups 13 (2008), no. 3-4, 495--505.
\end{thebibliography}
\end{document}